\documentclass[11pt,a4paper]{amsart}

\usepackage[T1]{fontenc}
\usepackage{microtype}
\usepackage[margin=32mm]{geometry}
\usepackage{amsmath,amssymb,amsthm,mathtools}
\usepackage[
  colorlinks=true,
  linkcolor=blue,
  citecolor=blue,
  urlcolor=blue
]{hyperref}
\hypersetup{
  pdftitle={A Refined Sum--Product Estimate via Higher Energies},
  pdfauthor={Kaiqiang Zhang, Yuyu Wang, Yuanguo Zeng}
}

\allowdisplaybreaks
\numberwithin{equation}{section}

\newtheorem{theorem}{Theorem}[section]
\newtheorem{proposition}[theorem]{Proposition}
\newtheorem{lemma}[theorem]{Lemma}

\theoremstyle{definition}

\theoremstyle{remark}
\newtheorem{remark}[theorem]{Remark}

\newcommand{\R}{\mathbb R}
\newcommand{\Rpos}{\mathbb R_{>0}}
\newcommand{\Rtimes}{\mathbb R^{\times}}
\newcommand{\E}{\mathrm E}
\newcommand{\ind}{\mathbf 1}

\title{A Refined Sum--Product Estimate via Higher Energies}
\author{Kaiqiang Zhang}
\author{Yuyu Wang}
\author{Yuanguo Zeng}
\date{}

\begin{document}

\begin{abstract}
Let $A\subset\R$ be a finite set. Combining the multiplicative slope estimate of
Rudnev--Stevens, Cushman's higher-energy regularization, Shakan's
$d^+$--$d^\times$ decomposition, and Solymosi's classical sum--product estimate,
we prove
\[
  |AA|^{204}|A+A|^{301}\gtrsim |A|^{675},
\]
where $\gtrsim$ suppresses a fixed polylogarithmic factor in $|A|$.
Consequently, for every $\varepsilon>0$,
\[
  \max\{|A+A|,|AA|\}
  \gg_{\varepsilon}
  |A|^{135/101-\varepsilon}.
\]
The proof is organized around two intermediate estimates. For every nonempty
finite set $B\subset\Rpos$,
\[
  d^\times(B)|BB|^{12}|B+B|^{16}\gtrsim |B|^{38},
\]
whereas for every nonempty finite set $U\subset\R$,
\[
  d^+(U)^{17}|U+U|^{29}\gtrsim |U|^{46}.
\]
\end{abstract}

\maketitle

\section{Introduction}
\label{sec:introduction}

For a finite set $A\subset\R$, write
\[
  A+A:=\{a+b:a,b\in A\},
  \qquad
  AA:=\{ab:a,b\in A\}.
\]
The sum--product problem over the reals asks for quantitative lower bounds for
$\max\{|A+A|,|AA|\}$ in terms of $|A|$. It originates in the work of
Erd\H{o}s and Szemer\'edi \cite{ErdosSzemeredi1983}. The ``few sums, many
products'' estimate of Elekes and Ruzsa \cite{ElekesRuzsa2003}, together with
Solymosi's geometric slope method \cite{Solymosi2009}, laid important foundations
for subsequent work on the real sum--product problem. In the developments that
broke the exponent $4/3$ arising from Solymosi's method, Konyagin and Shkredov
combined slope arguments, energy estimates, and tools of the small-sumset/large-product-set type \cite{KonyaginShkredov2015,KonyaginShkredov2016}; Rudnev and
Stevens later refined this framework further \cite{RudnevStevens2022}.

Higher energies, which extend the classical additive energy, were developed
systematically by Schoen and Shkredov and by Shkredov
\cite{SchoenShkredov2013,Shkredov2013Higher}. On the decomposition side, the
low-energy decomposition of Balog and Wooley \cite{BalogWooley2017} and the
higher-energy decomposition of Shakan \cite{Shakan2019} provide mechanisms for
separating additive and multiplicative structure. More recently, Cushman applied
higher-energy regularization to the real sum--product problem and proved that, for
every $\varepsilon>0$,
\[
  \max\{|A+A|,|AA|\}
  \gg_{\varepsilon}
  |A|^{\frac43+\frac{10}{4407}-\varepsilon}
\]
\cite{Cushman2025}. In 2026, Bloom, Sawin, Schildkraut, and Zhelezov disproved
the stronger near-quadratic conjecture over $\R$
\cite{BloomSawinSchildkrautZhelezov2026}. This does not, however, preclude
further improvements to the universal lower-bound exponent.

The present paper follows this line of ideas while keeping the additive and
multiplicative structural parameters separate until the final step. We combine
the multiplicative slope estimate of Rudnev--Stevens, Cushman's higher-energy
regularization, and Shakan's $d^+$--$d^\times$ decomposition to obtain the
following estimate.

\begin{theorem}
\label{thm:intro-main}
For every nonempty finite set $A\subset\R$,
\begin{equation}
  |AA|^{204}|A+A|^{301}
  \gtrsim
  |A|^{675}.
\label{eq:intro-mixed}
\end{equation}
In particular, for every $\varepsilon>0$,
\begin{equation}
  \max\{|A+A|,|AA|\}
  \gg_{\varepsilon}
  |A|^{135/101-\varepsilon}.
\label{eq:intro-epsilon}
\end{equation}
\end{theorem}

Since
\[
  \frac{135}{101}
  =\frac43+\frac1{303},
\]
the exponent in \eqref{eq:intro-epsilon} is strictly larger than
$\frac43+\frac{10}{4407}$. The proof of Theorem~\ref{thm:intro-main} is given in
Section~\ref{sec:main-theorem}.

The two main intermediate estimates are
\begin{equation}
  d^\times(B)|BB|^{12}|B+B|^{16}
  \gtrsim |B|^{38},
\label{eq:intro-mult-control}
\end{equation}
for every nonempty finite set $B\subset\Rpos$, and
\begin{equation}
  d^+(U)^{17}|U+U|^{29}
  \gtrsim |U|^{46},
\label{eq:intro-add-control}
\end{equation}
for every nonempty finite set $U\subset\R$. The first estimate is extracted from
the multiplicative slope argument of Rudnev--Stevens
\cite{RudnevStevens2022}; the second follows from Cushman's regularization lemma
\cite{Cushman2025}. Shakan's decomposition \cite{Shakan2019} then allows the two
estimates to be combined at the cost of a single structural loss.

\section{Preliminaries and notation}
\label{sec:preliminaries}

All sets in this paper are finite. For $X,Y\subset\R$, define
\[
  r_{X-Y}(t)
  :=\#\{(x,y)\in X\times Y:x-y=t\}.
\]
For $X,Y\subset\Rtimes$, define
\[
  r_{X/Y}(t)
  :=\#\{(x,y)\in X\times Y:x=ty\}.
\]
For $q\ge1$, set
\[
  \E_q^+(X,Y):=\sum_t r_{X-Y}(t)^q,
  \qquad
  \E_q^\times(X,Y):=\sum_t r_{X/Y}(t)^q,
\]
and abbreviate
\[
  \E_q^+(X):=\E_q^+(X,X),
  \qquad
  \E_q^\times(X):=\E_q^\times(X,X).
\]
We also write $\E^+=\E_2^+$ and $\E^\times=\E_2^\times$. These higher moments
may be viewed as extensions of the classical additive and multiplicative energies;
for systematic developments of higher-energy methods, see Schoen--Shkredov
\cite{SchoenShkredov2013} and Shkredov \cite{Shkredov2013Higher}.

Following Shakan \cite{Shakan2019}, for a nonempty set $X\subset\R$ define
\begin{equation}
  d^+(X)
  :=
  \sup_{\varnothing\ne Z\subset\R}
  \frac{\E_3^+(X,Z)}{|X||Z|^2}.
\label{eq:def-dplus}
\end{equation}
For a nonempty set $X\subset\Rtimes$, define
\begin{equation}
  d^\times(X)
  :=
  \sup_{\varnothing\ne Z\subset\Rtimes}
  \frac{\E_3^\times(X,Z)}{|X||Z|^2}.
\label{eq:def-dtimes}
\end{equation}
Whenever the corresponding quantity is defined,
\begin{equation}
  1\le d^+(X)\le |X|,
  \qquad
  1\le d^\times(X)\le |X|.
\label{eq:d-trivial}
\end{equation}

We write $X\ll Y$ if there is an absolute constant $C_0>0$ such that
$X\le C_0Y$. We write $X\lesssim Y$ if there are absolute constants
$C_0,C_1>0$ such that
\[
  X\le C_0(\log(2+N))^{C_1}Y,
\]
where $N$ denotes the cardinality of the ambient set in the statement under
consideration; the constants $C_0,C_1$ may vary from one occurrence to another.
The notation $X\gtrsim Y$ means $Y\lesssim X$. Thus fixed polylogarithmic losses
are suppressed throughout.

We shall repeatedly use the following moment interpolation inequality. Let
$(a_x)_x$ be a finite sequence of nonnegative real numbers and set
\[
  M_q:=\sum_x a_x^q.
\]
If $1\le q_0<q<q_1$ and $q=(1-\theta)q_0+\theta q_1$, then H\"older's
inequality gives
\begin{equation}
  M_q
  \le
  M_{q_0}^{1-\theta}M_{q_1}^{\theta}.
\label{eq:moment-interpolation}
\end{equation}
In particular,
\begin{align}
  M_{3/2}^{2/3}
  &\le M_1^{1/2}M_3^{1/6},
\label{eq:interp-3half-3}
\\
  M_{3/2}^{2/3}
  &\le M_1^{1/5}M_{12/7}^{7/15},
\label{eq:interp-3half-12seven}
\\
  M_2
  &\le M_{12/7}^{7/9}M_3^{2/9}.
\label{eq:interp-2}
\end{align}
These inequalities will be applied to the relevant additive representation
functions.

Finally, by Cauchy--Schwarz,
\begin{equation}
  \E^+(X)
  \ge
  \frac{|X|^4}{|X+X|},
  \qquad
  \E^\times(X)
  \ge
  \frac{|X|^4}{|XX|},
\label{eq:energy-lower}
\end{equation}
where the multiplicative inequality requires $0\notin X$. Although $\E^+(X)$ is
defined using differences, it counts the same additive quadruples as the second
moment of $r_{X+X}$. The analogous statement holds for multiplicative energy and
the product set.

\section{External inputs}
\label{sec:external-inputs}

We collect here the external results used below. The Rudnev--Stevens input is
stated in the precise dichotomous form needed for our argument. We also include
the reduction from their Proposition~1 to this form in order to make the range of
applicability explicit.

\begin{proposition}[Solymosi]
\label{prop:Solymosi}
Let $B\subset\Rpos$ be a nonempty finite set, and write
\[
  |B+B|=K|B|,
  \qquad
  |BB|=M|B|.
\]
Then Solymosi's multiplicative-energy estimate \cite{Solymosi2009} gives
\begin{equation}
  K^2M\gtrsim |B|.
\label{eq:Solymosi}
\end{equation}
\end{proposition}

\begin{proposition}[Consequence of Rudnev--Stevens]
\label{prop:RS}
Let $B\subset\Rpos$ be nonempty, put $m:=|B|$, and write
\[
  |B+B|=Km,
  \qquad
  |BB|=Mm,
  \qquad
  e:=\E^\times(B).
\]
Then at least one of the following two alternatives holds.

\begin{enumerate}
\item
\begin{equation}
  K^{16}M^{12}\gtrsim m^{10}.
\label{eq:RS-easy-case}
\end{equation}

\item There exist $\tau\ge1$, a set $\mathcal L\subset B/B$, and a subset
$\mathcal L'\subset\mathcal L$ such that
\begin{equation}
  \tau\le r_{B/B}(\lambda)<2\tau
  \qquad (\lambda\in\mathcal L),
\label{eq:RS-layer}
\end{equation}
\begin{equation}
  |\mathcal L|\tau^2\le e
  \lesssim |\mathcal L|\tau^2,
\label{eq:RS-energy-layer}
\end{equation}
and $|\mathcal L'|\gg|\mathcal L|$. If
\[
  B_\lambda:=B\cap\lambda^{-1}B,
\]
then, for every $\lambda\in\mathcal L'$,
\begin{equation}
  |BB_\lambda|
  \gtrsim
  \frac{m^6}{M^4K^8|\mathcal L|^{1/2}}
  =
  \frac{m^{18}}
       {|BB|^4|B+B|^8|\mathcal L|^{1/2}}.
\label{eq:RS}
\end{equation}
\end{enumerate}
\end{proposition}

\begin{proof}
If $m$ is smaller than a sufficiently large absolute constant, the assertion is
absorbed by adjusting the implicit constants. Hence assume that $m$ is large.
For any fixed $b_0\in B$, we have $b_0+B\subset B+B$ and $b_0B\subset BB$, so
$K,M\ge1$.

Decompose $B/B$ dyadically according to the size of $r_{B/B}$ and choose a layer
making a maximal contribution to $e=\E^\times(B)$. Thus there exist a dyadic
scale $\tau\ge1$ and a set $\mathcal L\subset B/B$ such that
\[
  \tau\le r_{B/B}(\lambda)<2\tau
  \qquad(\lambda\in\mathcal L),
\]
and
\[
  |\mathcal L|\tau^2\le e
  \ll (\log m)|\mathcal L|\tau^2.
\]
This is exactly \eqref{eq:RS-layer}--\eqref{eq:RS-energy-layer}.

We now check the nondegeneracy reductions used in the proof of Proposition~1 of
Rudnev--Stevens \cite[Proposition~1]{RudnevStevens2022}. Fix an absolute
constant $C_{\mathrm{RS}}>0$, sufficiently large for the multiplicity-scale
requirement in that proof. If $\tau\le C_{\mathrm{RS}}$, then
$|\mathcal L|\le|B/B|\le m^2$ gives $e\lesssim m^2$. On the other hand,
Cauchy--Schwarz yields
\[
  e\ge \frac{m^4}{|BB|}=\frac{m^3}{M}.
\]
Hence $M\gtrsim m$. Since $K\ge1$,
\[
  K^{16}M^{12}\gtrsim m^{12},
\]
so \eqref{eq:RS-easy-case} holds. We may therefore assume
$\tau>C_{\mathrm{RS}}$.

If $|\mathcal L|<m^{1/2}$, then $\tau\le m$ and
\eqref{eq:RS-energy-layer} imply
\[
  e\lesssim m^{5/2}.
\]
Combining this again with $e\ge m^3/M$ gives $M\gtrsim m^{1/2}$. Together with
\eqref{eq:Solymosi},
\[
  K^{16}M^{12}
  =(K^2M)^8M^4
  \gtrsim m^{10},
\]
so the first alternative still holds. We may therefore also assume
$|\mathcal L|\ge m^{1/2}$.

It remains to check the bunch parameter in the Rudnev--Stevens proof. They take
a sufficiently large absolute constant $C_*>128$ and use the integer
\[
  N_*:=\left\lceil C_*K^2Mm^{-1}\log m\right\rceil.
\]
Solymosi's original energy bound in fact gives
\[
  K^2M\gg \frac{m}{\log m}.
\]
Thus, if $C_*$ is chosen sufficiently large in advance, then $N_*>2$. If
$N_*\ge m^{1/2}$, then, for sufficiently large $m$,
\[
  C_*K^2Mm^{-1}\log m
  \ge m^{1/2}-1
  \ge \frac12m^{1/2},
\]
and hence
\[
  K^2M\gg \frac{m^{3/2}}{\log m}.
\]
Since $M\ge1$,
\[
  K^{16}M^{12}
  =(K^2M)^8M^4
  \gtrsim m^{12},
\]
which again places us in \eqref{eq:RS-easy-case}.

Consequently, if the first alternative fails, then necessarily
\[
  \tau>C_{\mathrm{RS}},
  \qquad
  |\mathcal L|\ge m^{1/2},
  \qquad
  2<N_*<m^{1/2}\le|\mathcal L|.
\]
These are precisely the nondegeneracy reductions made before Proposition~1 and at
the beginning of its proof in Rudnev--Stevens. Hence that proposition applies to
the present energy-rich layer $\mathcal L$. By
\cite[Proposition~1]{RudnevStevens2022}, there is a subset
$\mathcal L'\subset\mathcal L$ with
\[
  |\mathcal L'|\ge\frac18|\mathcal L|
\]
such that, for every $\lambda\in\mathcal L'$,
\[
  |BB_\lambda|
  \gg
  \frac{m^6}
       {M^4K^8|\mathcal L|^{1/2}(\log m)^7}.
\]
Absorbing the fixed power of $\log m$ into $\gtrsim$ gives \eqref{eq:RS}, and
also $|\mathcal L'|\gg|\mathcal L|$. This proves the dichotomy.
\end{proof}

\begin{proposition}[Cushman regularization]
\label{prop:Cushman}
There is an absolute constant $N_0$ with the following property. Let
$U\subset\R$ be finite with $N:=|U|\ge N_0$. Then there exists $B\subset U$ with
$b:=|B|\ge N/2$. Define
\begin{equation}
  \mathcal S
  :=
  \left\{
    y\in B+B:
    r_{B+B}(y)
    \ge
    \frac{b^2}{8|B+B|\log N}
  \right\},
\label{eq:popular-sums}
\end{equation}
and
\begin{equation}
  R
  :=
  \left\{
    x\in B:
    |(B+x)\cap\mathcal S|
    \ge \frac34 b
  \right\}.
\label{eq:rich-set}
\end{equation}
If
\[
  F:=\E_{12/7}^+(B),
\]
then there exist $\Delta\ge1$ and a nonempty set
\begin{equation}
  P_\Delta
  :=
  \{t:\Delta\le r_{R-R}(t)<2\Delta\}
\label{eq:PDelta}
\end{equation}
such that
\begin{equation}
  F\lesssim \Delta^{12/7}|P_\Delta|
\label{eq:regularisation}
\end{equation}
and
\begin{equation}
  \Delta^2|P_\Delta|^2b^2
  \ll
  \E_3^+(B)
  \sum_{p\in P_\Delta}r_{\mathcal S-\mathcal S}(p).
\label{eq:projection}
\end{equation}
\end{proposition}

\begin{proof}[Justification of the formulation]
This is a direct consequence of \cite[Lemma~1.7]{Cushman2025} in the form used
below. Cushman's lemma first chooses $B$ so that, up to at most a logarithmic
loss, the $12/7$-energy of the rich set $R$ controls $\E_{12/7}^+(B)$. A dyadic
decomposition then selects $P_\Delta$ so that
$\Delta^{12/7}|P_\Delta|$ captures this energy. All such logarithmic losses are
absorbed into $\lesssim$. The counting inequality in \eqref{eq:projection} is the
corresponding projection estimate from that lemma, with the number of solutions
written as $\sum_{p\in P_\Delta}r_{\mathcal S-\mathcal S}(p)$.
\end{proof}

\begin{theorem}[Shakan]
\label{thm:Shakan}
Let $A\subset\Rtimes$ be a nonempty finite set. Then there exist $X,Y\subset A$
such that
\begin{equation}
  X\cup Y=A,
  \qquad
  |X|,|Y|\ge \frac{|A|}{2},
  \qquad
  d^+(X)d^\times(Y)\lesssim |A|.
\label{eq:Shakan}
\end{equation}
\end{theorem}

\begin{proof}[Source]
This is the restriction of \cite[Theorem~1.10]{Shakan2019} to
$A\subset\Rtimes$.
\end{proof}

\section{Multiplicative control}
\label{sec:multiplicative-control}

\begin{lemma}
\label{lem:multiplicative-control}
If $B\subset\Rpos$ is nonempty, then
\begin{equation}
  d^\times(B)|BB|^{12}|B+B|^{16}
  \gtrsim |B|^{38}.
\label{eq:multiplicative-control}
\end{equation}
\end{lemma}

\begin{proof}
Set
\[
  m:=|B|,
  \qquad
  p:=|BB|,
  \qquad
  s:=|B+B|,
  \qquad
  \Pi:=BB,
  \qquad
  e:=\E^\times(B),
\]
and write $p=Mm$ and $s=Km$.

Apply Proposition~\ref{prop:RS}. If \eqref{eq:RS-easy-case} holds, then, since
$d^\times(B)\ge1$,
\[
  d^\times(B)p^{12}s^{16}
  =d^\times(B)m^{28}M^{12}K^{16}
  \gtrsim m^{38},
\]
and the conclusion follows immediately.

We may therefore assume that the second alternative of Proposition~\ref{prop:RS}
holds. For $\lambda\in\mathcal L'$, if $t\in BB_\lambda$, then $t\in\Pi$ and
$\lambda t\in\Pi$. Hence
\begin{equation}
  r_{\Pi/\Pi}(\lambda)
  \ge |BB_\lambda|.
\label{eq:quotient-lower}
\end{equation}
On the other hand, grouping the mixed multiplicative energy according to the
common quotient gives
\begin{equation}
  \E^\times(B,\Pi)
  =
  \sum_\lambda
  r_{B/B}(\lambda)r_{\Pi/\Pi}(\lambda).
\label{eq:mixed-energy-identity}
\end{equation}
Using \eqref{eq:RS}, \eqref{eq:quotient-lower},
\eqref{eq:mixed-energy-identity}, and \eqref{eq:RS-energy-layer}, we obtain
\begin{align}
  \E^\times(B,\Pi)
  &\gtrsim
  |\mathcal L|\tau
  \frac{m^{18}}{p^4s^8|\mathcal L|^{1/2}}
  \\
  &\gtrsim
  \frac{m^{18}e^{1/2}}{p^4s^8}.
\label{eq:mixed-lower-1}
\end{align}
By \eqref{eq:energy-lower}, $e\ge m^4/p$, and therefore
\begin{equation}
  \E^\times(B,\Pi)
  \gtrsim
  \frac{m^{20}}{p^{9/2}s^8}.
\label{eq:mixed-lower-2}
\end{equation}

Let $f(t):=r_{B/\Pi}(t)$. By Cauchy--Schwarz and the definition of
$d^\times(B)$,
\begin{align}
  \E^\times(B,\Pi)^2
  &=
  \left(\sum_t f(t)^2\right)^2
  \\
  &\le
  \left(\sum_t f(t)\right)
  \left(\sum_t f(t)^3\right)
  \\
  &\le
  (mp)\,d^\times(B)mp^2
  =d^\times(B)m^2p^3.
\label{eq:mixed-upper}
\end{align}
Combining \eqref{eq:mixed-lower-2} and \eqref{eq:mixed-upper} gives
\[
  \frac{m^{40}}{p^9s^{16}}
  \lesssim
  d^\times(B)m^2p^3,
\]
which is equivalent to \eqref{eq:multiplicative-control}.
\end{proof}

\section{Additive control}
\label{sec:additive-control}

\begin{lemma}
\label{lem:additive-control}
For every nonempty finite set $U\subset\R$,
\begin{equation}
  d^+(U)^{17}|U+U|^{29}
  \gtrsim |U|^{46}.
\label{eq:additive-control}
\end{equation}
\end{lemma}

\begin{proof}
Set
\[
  N:=|U|,
  \qquad
  Q:=|U+U|,
  \qquad
  D:=d^+(U).
\]
If $N$ is below the absolute threshold in Proposition~\ref{prop:Cushman}, then
the conclusion can be absorbed into the implicit constant. Hence assume that $N$
is sufficiently large and choose $B,\mathcal S,R,\Delta,P_\Delta$ as in
Proposition~\ref{prop:Cushman}. Put
\[
  b:=|B|,
  \qquad
  F:=\E_{12/7}^+(B).
\]
Then $b\asymp N$, $|B+B|\le Q$, and
\begin{equation}
  F\lesssim \Delta^{12/7}|P_\Delta|,
  \qquad
  \Delta^2|P_\Delta|^2b^2
  \ll
  \E_3^+(B)\Sigma,
\label{eq:reg-and-projection}
\end{equation}
where
\[
  \Sigma
  :=
  \sum_{p\in P_\Delta}r_{\mathcal S-\mathcal S}(p).
\]

We first estimate $\Sigma$. By the convolution identity
\[
  \Sigma
  =
  \sum_{x\in\mathcal S}r_{\mathcal S+P_\Delta}(x)
\]
and the definition of $\mathcal S$,
\begin{equation}
  \Sigma
  \lesssim
  \frac{Q}{b^2}
  \sum_x
  r_{B+B}(x)r_{\mathcal S+P_\Delta}(x).
\label{eq:sigma-popular}
\end{equation}
The identity
\begin{equation}
  \sum_x r_{B+B}(x)r_{\mathcal S+P_\Delta}(x)
  =
  \sum_x r_{B-P_\Delta}(x)r_{\mathcal S-B}(x)
\label{eq:convolution-identity}
\end{equation}
holds because both sides count quadruples $(b_1,b_2,s,p)$ satisfying
$b_1+b_2=s+p$. Hence H\"older's inequality gives
\begin{equation}
  \Sigma
  \lesssim
  \frac{Q}{b^2}
  \E_{3/2}^+(B,P_\Delta)^{2/3}
  \E_3^+(\mathcal S,B)^{1/3}.
\label{eq:sigma-holder}
\end{equation}

Since $B\subset U$,
\[
  \E_3^+(B,P_\Delta)
  \le
  \E_3^+(U,P_\Delta)
  \le
  DN|P_\Delta|^2.
\]
Also $\E_1^+(B,P_\Delta)=b|P_\Delta|$. Applying
\eqref{eq:interp-3half-3} to the sequence $r_{B-P_\Delta}$ and using
$b\asymp N$, we obtain
\begin{equation}
  \E_{3/2}^+(B,P_\Delta)^{2/3}
  \lesssim
  D^{1/6}N^{2/3}|P_\Delta|^{5/6}.
\label{eq:first-mixed-energy}
\end{equation}

We next estimate the second factor in \eqref{eq:sigma-holder}. By
\eqref{eq:popular-sums}, for every $x\in\R$,
\begin{align}
  r_{\mathcal S-B}(x)
  &\lesssim
  \frac{Q}{b^2}
  \sum_s r_{B+B}(s)\ind_B(s-x)
  \\
  &=
  \frac{Q}{b^2}
  \sum_t r_{B-B}(t)\ind_B(x+t).
\label{eq:triple-convolution}
\end{align}
The second line follows after the change of variables $t\mapsto -t$ and the
identity $r_{B-B}(t)=r_{B-B}(-t)$; equivalently, both expressions count the same
triples.

For every dyadic index $0\le i\le\lfloor\log_2 b\rfloor$ for which the
corresponding layer is nonempty, define
\[
  D_i
  :=
  \{t\in B-B:2^i\le r_{B-B}(t)<2^{i+1}\}.
\]
By Minkowski's inequality, then by $B\subset U$ and the definition of $D$,
\begin{align}
  \E_3^+(\mathcal S,B)^{1/3}
  &\lesssim
  \frac{Q}{b^2}
  \sum_i 2^i\E_3^+(B,D_i)^{1/3}
  \\
  &\le
  \frac{Q}{b^2}(DN)^{1/3}
  \sum_i2^i|D_i|^{2/3}.
\label{eq:second-mixed-pre}
\end{align}
For each $i$, put
\[
  A_i
  :=
  \sum_{t\in D_i}r_{B-B}(t)^{3/2}.
\]
Then $2^i|D_i|^{2/3}\le A_i^{2/3}$. Since there are at most $O(\log N)$
nonempty dyadic layers, H\"older's inequality yields
\begin{equation}
  \sum_i2^i|D_i|^{2/3}
  \lesssim
  \left(\sum_iA_i\right)^{2/3}
  =
  \E_{3/2}^+(B)^{2/3}.
\label{eq:dyadic-E32}
\end{equation}
Applying \eqref{eq:interp-3half-12seven} to $r_{B-B}$ and using
$\E_1^+(B)=b^2$, we get
\begin{equation}
  \E_{3/2}^+(B)^{2/3}
  \le
  b^{2/5}F^{7/15}.
\label{eq:E32-F}
\end{equation}
Substituting \eqref{eq:dyadic-E32}--\eqref{eq:E32-F} into
\eqref{eq:second-mixed-pre} and using $b\asymp N$ gives
\begin{equation}
  \E_3^+(\mathcal S,B)^{1/3}
  \lesssim
  D^{1/3}QN^{-19/15}F^{7/15}.
\label{eq:second-mixed-energy}
\end{equation}

Again, since $B\subset U$,
\begin{equation}
  \E_3^+(B)
  \le
  \E_3^+(U,B)
  \le
  DNb^2
  \lesssim DN^3.
\label{eq:E3B}
\end{equation}
Combining \eqref{eq:sigma-holder}, \eqref{eq:first-mixed-energy},
\eqref{eq:second-mixed-energy}, and \eqref{eq:E3B} with the second inequality in
\eqref{eq:reg-and-projection}, we obtain
\begin{equation}
  \Delta^2|P_\Delta|^{7/6}
  \lesssim
  D^{3/2}Q^2N^{-8/5}F^{7/15}.
\label{eq:key-delta}
\end{equation}
Indeed, after dividing by the factor $b^2\asymp N^2$ on the left, the total
exponent of $N$ is
\[
  3+\frac23-\frac{19}{15}-2-2
  =-\frac85.
\]

The first inequality in \eqref{eq:reg-and-projection} gives
\[
  F^{7/6}
  \lesssim
  \Delta^2|P_\Delta|^{7/6}.
\]
Together with \eqref{eq:key-delta}, this implies
\[
  F^{7/10}
  \lesssim
  D^{3/2}Q^2N^{-8/5},
\]
and hence
\begin{equation}
  F
  \lesssim
  D^{15/7}Q^{20/7}N^{-16/7}.
\label{eq:F-bound}
\end{equation}

Finally, applying \eqref{eq:interp-2} to $r_{B-B}$ and using
\eqref{eq:E3B} and \eqref{eq:F-bound}, we obtain
\begin{align}
  \E^+(B)
  &\le
  F^{7/9}\E_3^+(B)^{2/9}
  \\
  &\lesssim
  D^{17/9}Q^{20/9}N^{-10/9}.
\label{eq:E2-upper}
\end{align}
On the other hand, by \eqref{eq:energy-lower}, $b\asymp N$, and
$|B+B|\le Q$,
\begin{equation}
  \E^+(B)
  \gtrsim
  \frac{N^4}{Q}.
\label{eq:E2-lower}
\end{equation}
Comparing \eqref{eq:E2-upper} and \eqref{eq:E2-lower} gives
\[
  N^{46/9}
  \lesssim
  D^{17/9}Q^{29/9},
\]
which is equivalent to \eqref{eq:additive-control}.
\end{proof}

\section{The sum--product estimate}
\label{sec:main-theorem}

\begin{theorem}
\label{thm:main}
For every nonempty finite set $A\subset\R$,
\begin{equation}
  |AA|^{204}|A+A|^{301}
  \gtrsim |A|^{675}.
\label{eq:main-mixed}
\end{equation}
Consequently, for every $\varepsilon>0$,
\begin{equation}
  \max\{|A+A|,|AA|\}
  \gg_{\varepsilon}
  |A|^{135/101-\varepsilon}.
\label{eq:main-epsilon}
\end{equation}
\end{theorem}

\begin{proof}
First suppose that $A\subset\Rpos$. Set
\[
  n:=|A|,
  \qquad
  P:=|AA|,
  \qquad
  S:=|A+A|.
\]
By Theorem~\ref{thm:Shakan}, there exist $X,Y\subset A$ such that
\[
  |X|,|Y|\ge\frac n2,
  \qquad
  d^+(X)d^\times(Y)\lesssim n.
\]
Applying Lemma~\ref{lem:additive-control} to $X$ and using
$X+X\subset A+A$, we get
\begin{equation}
  d^+(X)^{17}S^{29}
  \gtrsim n^{46}.
\label{eq:add-X}
\end{equation}
Applying Lemma~\ref{lem:multiplicative-control} to $Y$ and using
$YY\subset AA$ and $Y+Y\subset A+A$, we get
\begin{equation}
  d^\times(Y)P^{12}S^{16}
  \gtrsim n^{38}.
\label{eq:mult-Y}
\end{equation}
Raising \eqref{eq:mult-Y} to the seventeenth power and multiplying by
\eqref{eq:add-X} yields
\begin{equation}
  \bigl(d^+(X)d^\times(Y)\bigr)^{17}
  P^{204}S^{301}
  \gtrsim n^{692}.
\label{eq:before-Shakan}
\end{equation}
Since $d^+(X)d^\times(Y)\lesssim n$, it follows that
\begin{equation}
  P^{204}S^{301}
  \gtrsim n^{675}.
\label{eq:positive-main}
\end{equation}

Now let $A\subset\R$ be an arbitrary nonempty finite set. The case $|A|=1$ is
immediate, so assume $|A|\ge2$. Let $A^*:=A\setminus\{0\}$ and choose the larger
of $A^*\cap\Rpos$ and $A^*\cap\R_{<0}$. If the negative part is chosen,
multiply it by $-1$. In this way we obtain a set $A'\subset\Rpos$ satisfying
\[
  |A'|\ge\frac{|A|}{4},
  \qquad
  |A'+A'|\le|A+A|,
  \qquad
  |A'A'|\le|AA|.
\]
Applying \eqref{eq:positive-main} to $A'$ proves \eqref{eq:main-mixed}.

Finally, let
\[
  M_0:=\max\{|A+A|,|AA|\}.
\]
By \eqref{eq:main-mixed},
\[
  M_0^{505}\gtrsim |A|^{675},
\]
so
\[
  M_0\gtrsim |A|^{675/505}
  =|A|^{135/101}.
\]
Since $\gtrsim$ suppresses only a fixed power of a logarithm, and since for every
fixed $C$ and every $\varepsilon>0$ one has
$(\log(2+|A|))^C\ll_{C,\varepsilon}|A|^\varepsilon$, this gives
\eqref{eq:main-epsilon}.
\end{proof}

\begin{remark}
\label{rem:exponent-bookkeeping}
The final exponents are completely determined by the two control lemmas and
Shakan's decomposition:
\[
  204=17\cdot12,
  \qquad
  301=29+17\cdot16,
  \qquad
  675=46+17\cdot38-17.
\]
Thus, once Lemmas~\ref{lem:multiplicative-control} and
\ref{lem:additive-control} are established, there is no further freedom in the
exponent bookkeeping leading to Theorem~\ref{thm:main}.
\end{remark}

\section*{Acknowledgments}

The authors thank Peng Yang for helpful discussions and valuable comments, and Ilya Shkredov for insightful historical remarks on the development of higher-energy methods and the evolution of sum--product estimates beyond Solymosi's bound.

\end{document}